\documentclass{amsart}

\usepackage{latexsym,amsmath,amssymb,amsfonts,amscd,graphics,appendix,amsxtra}
\usepackage[mathscr]{eucal}
\usepackage[all]{xypic}
\usepackage[normalem]{ulem}

\numberwithin{equation}{section}
\theoremstyle{plain}
\newtheorem{theorem}[equation]{Theorem}

\newtheorem{lemma}[equation]{Lemma}

\newtheorem{proposition}[equation]{Proposition}
\newtheorem{corollary}[equation]{Corollary}

\theoremstyle{remark}
\newtheorem{remark}[equation]{Remark}

\theoremstyle{definition}

\def\Q{\mathbb{Q}}

\newcommand{\bQ}{\mathbb{Q}}
\newcommand{\bZ}{\mathbb{Z}}

\newcommand{\bC}{\mathbb{C}}

\newcommand{\GL}{\mathrm{GL}}

\newcommand{\SL}{\mathrm{SL}}

\title[]{Bhargava Integer Cubes and Weyl Group Multiple Dirichlet Series}

\author{Jun Wen}
\email{jwen@math.umass.edu}
\address{Department of Mathematics and Statistics, University of Massachusetts Amherst}

\begin{document}
\bibliographystyle{alpha}

\maketitle
\begin{abstract}
We investigate the Shintani zeta function associated to a prehomogeneous vector space. The example under consideration is the set of $2 \times 2\times 2$ integer cubes. We show that there are three relative invariants under a certain parabolic group action, they all have arithmetic meanings and completely determine the equivalence classes. We show that the associated Shintani zeta function coincides with the $A_3$ Weyl group multiple Dirichlet series. Finally, we show that the set of semi-stable integer orbits maps finitely and surjectively to a certain moduli space. 
\end{abstract}

\section{Introduction}

The study of Dirichlet series in several complex variables has seen
 much development in recent years. Multiple Dirichlet series (MDS) can arise
 from metaplectic Eisenstein series, zeta functions of prehomogeneous
 vector spaces (PVS), height zeta functions or multiple zeta values.  This
 list is far from exhaustive.  In general, completely different
 techniques are involved in the study of different series arising in
 these different contexts.  Partially for this reason, it is of great
 interest to identify examples of multiple Dirichlet series which lie
 in two or more of the areas above.
 
This paper investigates one such example. We study in detail a
 Shintani zeta function associated to a certain prehomogeneous vector
 space and show that it coincides with a {\em Weyl group multiple Dirichlet
   series} (WMDS) of the form introduced in BBCFH (\cite{BBCFH}). This latter series (in
 three complex variables) is also
 a Whittaker function of a Borel Eisenstein on the metaplectic double
 cover of $\GL_4$.
 
The seminal work of Bhargava \cite{bhargava} generalizing Gauss's composition law on binary quadratic forms starts with investigating the rich structure of integer orbits of $2 \times 2 \times 2$-cubes acted on by $\SL_2(\bZ) \times \SL_2(\bZ)\times \SL_2(\bZ)$, which refers to the $D_4$ case among the list of classification of PVS given in Sato-Kimura \cite{satokimura}. Bhargava shows that the set of projective integer orbits with given discriminant has a group structure and it is isomorphic to the square product of narrow class group of the quadratic order with that discriminant. 

We begin with some definitions. 
Let $G$ be a connected complex Lie group, usually $G$ is a complexification of a real Lie group. A prehomogeneous vector space (PVS) $V$ of $G$, denoted by $(G, V)$, is a complex finite dimensional vector space $V$ together with a holomorphic representation of $G$, such that $G$ has an open dense orbit in $V$. Let $P$ be a complex polynomial function on $V$. We call it a relative invariant of $G$ if $P(g v) =\chi(g) P(v)$ for some rational character $\chi$ of $G$. 
We say the PVS has $n$ relative invariants if $n$ algebraically independent relative invariant polynomials generate the invariant ring. We define the set of semi-stable points $V^{ss}$ to be the subset on which no relative invariant polynomial vanishes. 

Let $V(\bZ) = \bZ^2 \otimes \bZ^2 \otimes \bZ^2$ and consider the action of $ B_2(\bZ) \times B_2(\bZ) \times \SL_2(\bZ)$, where $B_2(\bZ)$ is the subgroup of lower-triangular matrices in $\SL_2(\bZ)$ with positive diagonal elements. We denote $A \in V(\bZ)$ by $\left( \left( \begin{matrix}
a & b  \\
c & d \end{matrix} \right),  \left( \begin{matrix}
e & f  \\
g & h \end{matrix} \right) \right)$ for simplicity. The complex group $ B_2(\bC) \times B_2(\bC) \times \GL_2(\bC)$ acting on $V(\bC)$ is again a PVS and will have three relative invariants. They are $ m(A) = ad- bc, n(A)= ag - ce$ and the discriminant $D(A) =\mathrm{disc}(A)$. 

The primary object of study of PVS is the Shintani zeta function, see \cite{satoshintani} for the introduction and \cite{shintani} for the application to the average values of $h(d)$, the number of primitive inequivalent binary quadratic forms of discriminant $d$. The Shintani zeta function in three variables associated to the PVS of $2 \times 2 \times 2$-cubes is defined to be:
$$Z_{\rm{Shintani}}(s_1, s_2, w)  =  \sum_{ A \in \left( B_2(\bZ) \times B_2(\bZ) \times \SL_2(\bZ) \right) \backslash V^{ss}(\bZ)} \frac{1}{|D(A)|^w |m(A)|^{s_1} |n(A)|^{s_2}} .$$
We denote the partial sum of Shintani zeta function by 
\begin{align*}
Z_{\rm{Shintani}}^{\rm{odd}}(s_1, s_2, w)  =  \sum_{\substack{ A \in \left( B_2(\bZ) \times B_2(\bZ) \times \SL_2(\bZ) \right) \backslash V^{ss}(\bZ) \\ D(A) \ \rm{odd}}} \frac{1}{|D(A)|^w |m(A)|^{s_1} |n(A)|^{s_2}} .
\end{align*}
We will show that it is closely related to another multiple Dirichlet series which arises in the Whittaker expansion of the Borel Eisenstein series on the metaplectic double cover of $\GL_4$, which is found to be the WMDS associated to the root system $A_3$ and it is of the form:
\begin{align*}
Z_{\rm{WMDS}} (s_1, s_2, w)  =  \sum_{ D \  \mathrm{odd} \ \mathrm{discriminant} } \frac{1}{|D|^w} \sum_{m, n>0} \frac{\chi_D(\hat{m}) \chi_D(\hat{n})}{m^{s_1} n^{s_2}}a(D, m, n)
\end{align*}
where $\hat{m}$ denotes the factor of $m$ that is prime to the square-free part of $D$ and $\chi_D$ is the quadratic character associated to the field extension $\bQ(\sqrt{D})$ of $\bQ$. A precise formula of the coefficients $a(D, m, n)$ will be given in section 4. 

The main results of this paper, given in Theorems 1.1, 1.2 and 1.3 below, are as follows:
\begin{itemize}
\item[(1)]We give an explicit description of the Shintani zeta function associated to the PVS of $2 \times 2 \times 2$-cubes. 
\item[(2)]We show how the series is related to the quadratic WMDS associated to the root system $A_3$. 
\item[(3)]We give an arithmetic meaning to the semi-stable integer orbits of the PVS. 
\end{itemize}

Denote by $A(d, a)$ the number of solutions to the congruence $x^2 =d \ (\text{mod} \ a) $.
\begin{theorem}
The Shintani zeta function associated to the $\mathrm{PVS}$ of $2\times 2 \times 2$-cubes is given by:
\begin{align*}
Z_{\rm{Shintani}}(s_1, s_2, w)  &=  \sum_{ A \in \left( B_2(\bZ) \times B_2(\bZ) \times \SL_2(\bZ) \right) \backslash V^{ss}(\bZ)} \frac{1}{|D(A)|^w |m(A)|^{s_1} |n(A)|^{s_2}} \\
& =\sum_{D = D_0 D_1^2\neq 0} \frac{1}{|D|^w} \sum_{m, n > 0} \frac{B(D, m,n)}{m^{s_1}n^{s_2}}, \end{align*}
where
$$B(D, m, n) =\sum_{\substack{d|D_1 \\ d|m,\ d|n}} d \cdot A(D/ d^2, 4m/d) \cdot A(D/d^2, 4n/d).$$
\end{theorem}

\begin{theorem}
The Shintani zeta function can be related to the $A_3$-$\mathrm{WMDS}$ in the following way:
\begin{align*}
Z^{\mathrm{odd}}_{\rm{Shintani}}(s_1, s_2, w) = & 4  \frac{\zeta(s_1)}{\zeta(2 s_1)} \frac{\zeta(s_2)}{\zeta(2 s_2)}   
Z_{\rm{WMDS}} (s_1, s_2, w).    
\end{align*}
where $\zeta(s)$ stands for the Riemann zeta function. 
\end{theorem} 

Lastly, we give the arithmetic meaning to the semi-stable integer orbits of the PVS by showing that:
\begin{theorem}
There is a natural map, which is a surjective and finite morphism,
$$\left( B_2(\bZ) \times B_2(\bZ) \times \SL_2(\bZ) \right)\backslash V^{ss}(\bZ) \to \mathrm{Iso} \backslash \{(R;I_1, I_2): R/I_1 \cong \mathrm{N}(I_1)\bZ, R/I_2 \cong \mathrm{N}(I_2)\bZ \},$$
where $R$ is an oriented quadratic ring and $I_i's$ are the oriented ideals in $R$ with the norm $\mathrm{N}(I_i)$. The cardinality $ n( R; I_1, I_2)$ of the fiber is equal to
$$\sigma_1 \left( g.c.d.(D_1, a_1, a_2)\right),$$
where $D =D_0 D_1^2= \mathrm{disc}(R)$ with $D_0$ is square-free, and $a_i=\mathrm{N}(I_i)$.
It further satisfies
$$\sum _{\substack{(R; I_1, I_2)/\sim \\ \mathrm{N}(I_i)=a_i}} n (R; I_1, I_2)  = B(D, a_1, a_2).$$
\end{theorem}

The organization of this paper is as follows. In section $2$, we will review a double Dirichlet series arising from three different approaches. We show the connection between Shintani's PVS approach and the $A_2$-WMDS approach. In section $3$, we will investigate the structure of integer orbits of the PVS of $2\times 2 \times2$-cubes, using the reduction theory, to derive the main formula in Thm 1.1. In section $4$, we will show its connection to the $A_3$-WMDS. The main idea of the proof is to consider the generating function for the $p$-parts of $a(D, m,n)$. From the construction of an $A_3$-WMDS, the $p$-parts of the rational function which is invariant under the Weyl group action is given explicitly by taking the residue of the convolution of two rational functions of $A_2$ root system. We show it coincides with the generating function of $a(D, m,n)$. In the section $5$, we will show the set of semi-stable integer orbits naturally encodes the arithmetic information by showing it maps finitely and surjectively to the moduli space of isomorphism classes of pairs $(R; I_1, I_2)$, where $R$ is an oriented quadratic ring and $I_i's$ are the oriented ideals in $R$. We will recall the definition of orientation of a quadratic ring and its ideal. We further show that each fiber is counted by a divisor function. 

\subsection*{Acknowledgements}
The author would like to thank Gautam Chinta and Takashi Taniguchi for useful discussions. He is  grateful to Gautam Chinta for making valuable comments. He also thanks the organizers of the ICERM Semester Program on Automorphic Forms, Combinatorial Representation Theory and Multiple Dirichlet Series from January to May of 2013, where some of this work was carried out.

\section{$A_2$ Weyl Group Dirichlet series}
In this section we will introduce three double Dirichlet series and show that they are all essentially the same. The results are not new, Diamantis and Goldfeld in \cite{diamantisgoldfeld} show their relations using a type of converse theorem, and Shintani in \cite[\S 2]{shintani} compares his series to Siegel's, but the computations in section will serve as a prototype for the more involved computations involving 3-variable Dirichlet series in the latter sections. 
\subsection{Siegel Double Dirichlet Series}
The quadratic multiple Dirichlet series first appeared in the paper of  Siegel \cite{siegel}, and its twisted Euler product with respect to one variable was given explicitly. Siegel constructed his series as the Mellin transform of an Eisenstein series of half integral weight on the congruence subgroup $\Gamma_0(4)$. In this section we will first recall the definition of Siegel's double Dirichlet series and then, using the multiplicative property of Euler products, prove it can be expressed as a sum formed from quadratic characters. 

Denote by $A(d, a)$ the number of solutions to the quadratic congruence equation $x^2 = d \ ( \mathrm{mod} \  a)$.
Then Siegel's double Dirichlet series is defined to be:
$$ Z_{\text{Siegel}} (s, w) =  \sum_{d\neq 0} \frac {1}{|d|^w}\sum_{a \neq 0}\frac{A(d, a)}{|a|^s}.$$

For any positive prime integer $p$ and any integer $d$, we define a generating series for the $p$-part of the inner summation of above series to be
\begin{align}
&f_p(d, s)= (1-p^{-s})\sum^{\infty}_{l=0}A(d,p^l) p^{-ls}  \ (p \neq 2) ,\notag \\
&f_2(d, s)= (2^s-1) \sum^{\infty}_{l=1}A(d,2^l)2^{-ls}. \notag
\end{align}
Siegel shows in \cite[\S 4]{siegel} that 
\begin{align*}
\frac{1- \chi_d(p)p^{-s}}{1-p^{-2s}}f_p(d, s)= \begin{cases} p^{\alpha(1-2s)}+ \left(1-\chi_d(p)p^{-s}\right) \sum^{\alpha-1}_{l=0}p^{l(1-2s)} \  &\text{($p\neq2$)}, \\
\frac{1+\chi_d(2)}{1+2^{-s}}+ \left(2^{1-s}-\chi_d(2)\right)\sum^{\alpha}_{l=0}2^{l(1-2s)} \  &\text{($p = 2$)},
 \end{cases}
\end{align*}
where $\chi_d$ is the quadratic character associated to the field $\mathbb{Q}(\sqrt{d})$ and $ p^ {2\alpha}$ is the highest power of $p$ which divides $d/d^{*}$, $d^{*}$ is the fundamental discriminant of the field $\mathbb{Q}(\sqrt{d})$. Now the inner summation of Siegel's double Dirichlet series can be expressed as a normalized quadratic L-function.

\begin{proposition}
Fix an integer $d \neq 0$, then
$$\sum_{a < 0}\frac{ A(d,a)}{|a|^{s}} = \sum_{a > 0}\frac{ A(d,a)}{a^{s}}  = \zeta(2s)^{-1} \zeta (s) L(s, \chi_d) P(d, s) $$
where the last term is 
\begin{align}
P(d,s) =& 2^{-s} \prod_{p \neq 2} \left( p^{\alpha(1-2s)} + (1- \chi_d(p) p^{-s} )\sum^{\alpha-1}_{l=0}p^{l(1-2s)}\right) \notag \\
&\cdot \left( \frac{1+\chi_n(2)}{1+2^{-s}}+\frac{1-\chi_d(2) 2^{-s}}{1+2^{-2s}}(2^{s}-1)+
 \left(2^{1-s}-\chi_n(2) \right)\sum^{\alpha}_{l=0}2^{l(1-2s)} \right). \notag
\end{align}
\end{proposition}
\begin{proof}
The first equality of the claim follows from the fact that
$$A(d, a) = A(d, -a).$$
For the second equality, first note that the multiplicative property holds
$$A(d,m) A(d,n) = A(d,mn)$$ 
for any pairs of coprime positive integers $m$ and $n$, by the Chinese remainder theorem.

Then the results follows from two equations
$$\frac{1- \chi_n(p)p^{-s}}{1-p^{-2s}}\left(1-p^{-s}\right) \sum_{l=0}^{\infty} A\left(d,p^l\right) p^{-ls} =p^{\alpha(1-2s)}+ \left(1-\chi_n\left(p\right)p^{-s}\right) \sum^{\alpha-1}_{l=0}p^{l\left(1-2s\right)} ,$$
$$\frac{1- \chi_n(2)2^{-s}}{1-2^{-2s}}\left(2^{s}-1\right) \sum_{l=1}^{\infty} A\left(d,2^l\right) 2^{-ls} =\frac{1+\chi_n(2)}{1+2^{-s}}+ \left(2^{1-s}-\chi_n(2)\right)\sum^{\alpha}_{l=0}2^{l(1-2s)} ,$$
as well as the multiplicative property.
\end{proof}

\subsection{Shintani Double Dirichlet Series}
Another approach to the theory of double Dirichlet series is based on the zeta function associated to the prehomogeneous vector space (PVS) developed by M. Sato and Shintani in \cite{satoshintani}. In \cite{shintani}, where the double Dirichlet series associated to the PVS of binary quadratic forms acted on by the Borel subgroup of $\GL_2(\bC)$ is studied in detail, the author obtains the mean values of class numbers of primitive and integral binary quadratic forms. It should be mentioned that in \cite{Hsaito} the essentially same double Dirichlet series is discovered as the zeta function associated to another PVS. In this section, we will recall the Shintani zeta function in two variables arising from the PVS approach. 

Now we let $B_2(\bC)$ be the subgroup of lower-triangular matrices in $\GL_2({\bC})$ and let $\rho$ be the representation of $\GL_2({\bC})$ acting on the three dimensional vector space $U(\bC)=\{Q(u, v)= a u^2 +b uv+ cv^2|(a, b, c )\in \bC^3\}$ of binary quadratic forms 
as follows
$$\rho(g) (Q)(u, v)  = Q(au +cv , bu+dv)$$
where $ g =\left( \begin{array}{cc}
a & b  \\
c & d \end{array} \right)$. It is well known that $(B_2(\bC), U(\bC))$ is a PVS and there are two relative invariants for the action of $ B_2(\bC)$ on $U(\bC)$, namely, the discriminant $\mathrm{disc}(Q)=b^2 -4 ac$ of the quadratic form and $a= Q(1, 0)$. These two invariants freely generate the ring of relative invariants. 

Define $B_2(\bZ)$ to be the Borel subgroup in $\SL_2(\bZ)$ with positive diagonal elements. The Shintani zeta function associated to the PVS $(B_2({\bC}), U(\bC))$ is defined to be
\begin{align*}
Z_{\mathrm{Shintani}}( s, w;B_2)& = \sum_{Q \in B_2(\bZ) \backslash U^{ss}(\bZ)}\frac{1}{|Q(1, 0)|^s |\mathrm{disc}(Q)|^w} \notag \\
&=\sum_{a\neq 0} \frac{1}{|a|^{s}} \sum_{\substack{0 \leq b \leq 2a -1 \\ c: b^2 - 4a c \neq 0}} \frac{1}{|b^2 - 4 a c|^w }. \notag 
\end{align*}
where $U^{ss}(\bZ)$ is the semi-stable subset of (not necessarily primitive) quadratic forms with $Q(1, 0) \neq 0$ and non-zero discriminant.  
Alternatively, we can express the Shintani zeta function as 
$$ \sum _{Q \in \SL_2(\bZ) \backslash U^{ss}(\bZ)} \frac{1}{|\mathrm{disc}(Q)|^{w}} \sum _{\gamma \in B_2(\bZ) \backslash \SL_2(\bZ) /\mathrm{stab}_{Q}}\frac{1}{|\gamma \circ Q(1, 0)|^{s}}.$$
For the rest of this section we suppress the $B_2$ from the notation. 
We also define 
\begin{align*}
Z_{\mathrm{Shintani}}^{\mathrm{odd}}( s, w)& = \sum_{\substack{Q \in B_2(\bZ)  \backslash U^{ss}(\bZ) \\ \mathrm{disc}(Q) \ \mathrm{odd}}} \frac{1}{|Q(1, 0)|^s |\mathrm{disc}(Q)|^w} \notag \\
&=\sum_{a\neq 0} \frac{1}{|a|^{s}} \sum_{\substack{0 \leq b \leq 2a -1 \\ c: b^2 - 4a c \ \mathrm{odd}}} \frac{1}{|b^2 - 4 a c|^w }. \notag 
\end{align*}

\begin{remark}
In section 5, we will show that the orbits in $B_2(\bZ) \backslash U^{ss}_{\rm{primitive}}(\bZ)$ parameterize the isomorphism classes of the pairs $(R, I)$, where $R$ is an oriented quadratic ring and $I$ is an oriented ideal with cyclic quotient in $R$. We will call $(R_1, I_1)$ and $(R_2, I_2)$ isomorphic if there is a ring isomorphism from $R_1$ to $R_2$ preserving the orientation and sending $I_1$ to $I_2$.  
\end{remark}

\begin{lemma}
The Shintani zeta function can be written as
$$Z_{\rm{Shintani}}( s, w)= \xi_1(s, w) + \xi_2(s, w),$$
where
$\xi_i(s, w)= \sum_{a, d > 0 } \frac{A( (-1)^{i-1}d,4a)}{a^s d^w} .$
\end{lemma}
\begin{proof}
First note that under the action of $g= \left( \begin{array}{cc} 1 & 0\\ m &1 \end{array} \right)$ on the quadratic form $Q(u, v)  = a u^2 + buv + cv^2$, the middle coefficient $b$ is mapped to $b+ 2 a m$. Given non-zero integers $a$ and $d$, the number of the solutions to $b^2 - 4a c = d$ with $0 \leq d \leq 2a-1$ and $ c\in \bZ$ is equal to $\frac{A(d, 4a)}{2}$.  So we have the equality
$$Z_{\rm{Shintani}}( s, w)  = \frac{1}{2} \sum_{a,d  \neq 0} \frac{A(d, 4a) }{|a|^s |d|^w}.$$
On the other hand, $A(d, 4a) = A(d, -4a) $. Therefore,
$$Z_{\rm{Shintani}}( s, w) =\sum_{a,d  > 0} \frac{A(d, 4a) }{a^s d^w} + \sum_{a,d  > 0} \frac{A(-d, 4a) }{a^s d^w}.$$
The result follows. 
\end{proof}
As a corollary to Proposition 2.1, we can express the inner summation of the Shintani zeta function in terms of a quadratic Dirichlet $L$-function. 
\begin{corollary}
Fix $d \neq 0$. Then we have
$$\sum_{a >0}\frac{ A(d,4a)}{a^{s}} = \zeta(2s)^{-1} \zeta (s) L(s, \chi_d) P'(d, s), $$
where the last term is
\begin{align}
P'(d,s) =& 4^{s} \prod_{p \neq 2} \left( p^{\alpha(1-2s)} + \left(1- \chi_d(p) p^{-s} \right)\sum^{\alpha-1}_{l=0}p^{l(1-2s)} \right)\notag \\
&\cdot \left(\frac{1+\chi_n(2)}{1+2^{-s}}+ \left(2^{1-s}-\chi_n(2)\right)\sum^{\alpha}_{l=0}2^{l(1-2s)} - 
  \frac{1- \chi_n(2)2^{-s}}{1-2^{-2s}}\left(1-2^{-s}\right)\right) .\notag
\end{align}
\end{corollary}
\begin{proof}
As $\sum_{a>0} \frac{A(d, 4a)}{a^s} = 4^s \sum_{a >0}\frac{ A(d,4a)}{(4a)^{s}}$, by the proof of proposition 2.1, we need only to correct the generating function at prime $p=2$ in order to incorporate the factor $4$, in which case the generating function should be
$$ \sum_{l=0}^{\infty} A(d,4 \times 2^l) 2^{-ls}= 4^s \sum_{l=2}^{\infty} A(d, 2^l) 2^{-ls}.$$
While the function on the right hand side satisfies
\begin{align}
\frac{1- \chi_d(2)2^{-s}}{1-2^{-2s}}(2^{s}-1) \sum_{l=2}^{\infty} A(d,2^l) 2^{-ls}=  & \frac{1- \chi_d(2)2^{-s}} {1- 2^{-2s}} f_2(d, s) \notag \\
& - \frac{1- \chi_n(2)2^{-s}}{1-2^{-2s}}(2^{s}-1) A(d, 2) 2^{-s}  .\notag 
\end{align}
Note that $A(d, 2) =1$, therefore, using the multiplicative property of $A(d,  \cdot)$, we have
$$\sum_{a>0} \frac{A(d, 4a)}{a^s} = 4^s \sum_{a >0}\frac{ A(d,4a)}{(4a)^{s}}  =  \zeta(2s)^{-1} \zeta (s) L(s, \chi_d) P'(d, s).$$
The result follows. 
\end{proof}

\subsection{$A_2$-Weyl Group Multiple Dirichlet Series}
Weyl group multiple Dirichlet series are a class of multiple Dirichlet series coming from Eisenstein series on metaplectic groups.  The simplest example is the quadratic $A_2$ Weyl group double Dirichlet series (see \cite{chintagunnells}),  it is defined as
$$ Z_{A_2}(s,w) =\sum_{\substack{ m > 0,  \\ D \ \mathrm{odd} \  \mathrm{discriminant}}} \frac{\chi_D(\hat{m})}{m^s |D|^w}a(D, m), $$
where $\hat{m}$ is the factor of $m$ that is prime to the square-free part of $D$ and $\chi_D$ is the quadratic character associated to the field extension $\Q(\sqrt{D})$ of $\bQ$. Moreover, the multiplicative factor $a(D, m)$ is defined by
$$a(D,m) = \prod_{ p^k||D,p^l||m}a(p^k, p^l)$$ 
and
$$a(p^k, p^l) = \begin{cases} \min(p^{k/2},p^{l/2})  & \mbox{if}\  \min(k,l)\  \mbox{is even}, \notag \\
   0 & \mbox{otherwise}. \notag \end{cases}$$
Then the relation between Shintani's zeta function $Z(s, w)$ and $Z_{A_2}(s, w)$ is implied by
\begin{proposition}
Fix $D$ an odd discriminant. Then
$$  \sum_{m>0} \frac{A(D, 4m)}{ m^s} =2  \frac{\zeta(s)}{\zeta(2s)} \sum_{m>0} \frac{\chi_D(\hat{m}) a(D, m) }{ m^s}.$$ 
\end{proposition}
\begin{proof}
Analogous to the $p$-part formula of $a(p^k, p^l)$, we will show that for an odd prime $p$ 
$$A(p^k, p^l)  = \begin{cases} 2 a(p^k, p^l) &\text{if} \ k<l, \\ p^{\lfloor l /2 \rfloor} &\mbox{otherwise}. \end{cases}$$
First consider the case when $ p \neq 2 $.  If $k<l$ and $k$ is an odd integer, then the congruence equation $x^2= p^k (\text{mod} \ p^l) $  reduces to the equation of $x^2 =  p  (\text{mod} \ p^i) $  for some power $i$ of $p$ and there is no solution to it; while when $k$ is an even integer, the congruence equation $x^2 = p^k ( \text{mod} \ p^l)$ reduces to the equation of or $ x^2 =  1  (\text{mod} \ p^i)$ and there are two solutions to it. In both case, the number of solutions are both equal to the value of $2 a(p^k, p^l)$ by its definition.

If on the other hand $k \geq l$, then the set of solutions to the congruence equation $x^2 = p^k (\text{mod} \ p^l)$ is the set of multipliers of $p^{\left\lceil \frac{l}{2} \right\rceil}$, so the number of distinct solutions mod $p^l$ is $p^{\left\lfloor \frac{l}{2} \right\rfloor}$.

Therefore, for odd prime $p$,
$$ A(p^k,p^l ) =\chi_{p^k} (\hat{p}^l) a(p^k,p^l) + \chi_{p^k} (\hat{p}^{l-1}) a(p^k, p^{l-1}) =a(p^k,p^l) + a(p^k, p^{l-1}),  $$
where we set the term $ a(p^k, p^{l-1})$ equal to $0$ when $l=0$.

Next, using Hensel's lemma, an integer $d$ relatively prime to an odd prime $p$ is a quadratic residue modulo any power of $p$ if and only if it is a quadratic residue modulo $p$. In fact, if an integer $d$  is prime to the odd prime $p$, as
$$A(d, p^l) =2 \iff  \chi_d(p) =1 \iff A(d, p) =2, $$ 
so
$$A(d, p^l) = \chi_d(p^l) + \chi_d(p^{l-1}). $$

By the prime power modulus theory \cite{gauss}, if the modulus is $p^l$,
then $p^kd$ is a
$$\begin{cases}
\text{quadratic residue modulo} \ p^l\  \text{if} \  k \geq l, \\
 \text{non-quadratic residue modulo}\  p^l\  \text{if}\  k < l \ \text{is odd}, \\
\text{quadratic residue modulo} \ p^l \ \text{if}\  k < l \ \text{ is even and}\  d \ \text{is a quadratic residue},\\
\text{non-quadratic residue modulo} \ p^l\  \text{if}\  k < l \ \text{ is even and otherwise}.
\end{cases}$$
Therefore, for an odd integer $d$ prime to $p\neq 2$, we have
$$A(dp^k, p^l) = \begin{cases}  0 & \chi_{d}(p^l) = \mbox{-1 and} \  k<l \ \text{even} ,\\  A(p^k,p^l) &\mbox{otherwise}. \end{cases}$$
In the former case, we have:
\begin{equation*}\tag{1}
A(dp^k, p^l) =0= \chi_{d p^k}(\hat{p}^l) a(dp^k, p^l) + \chi_{d p^k}(\hat{p}^{l-1}) a(dp^k, p^{l-1}).
\end{equation*}
In the latter case, we also have:
\begin{equation*} \tag{2}
A(dp^k, p^l)  = A(p^k, p^l)= \chi_{d p^k}(\hat{p}^l) a(dp^k, p^l) + \chi_{d p^k}(\hat{p}^{l-1}) a(dp^k, p^{l-1}).
\end{equation*}

Now let $D$ be an arbitrary odd discriminant. Given an prime integer $p$, write $D= D_0 p^k$, where $D_0$ is prime to $p$. Then from the equality $(1)$ and $(2)$ with $d$ replaced by $D_0$,
it follows that
$$ \sum_{l=0}^{\infty} A(D, p^l) p^{-ls}  =(1-p^{-2s}) (1-p^{-s})^{-1} \sum_{l=0}^{\infty} \chi_{D}(\hat{p}^l) a(D, p^l) p^{-ls}.$$  

For $p=2$, we define $\tilde{P}_2(D, s)$ by equating
\begin{align} \sum_{l =0}^{\infty} \frac{A(D,  2^{l+2})}{2^{ls} } &=\tilde{P}_2(D, s)  (1- 2^{-2s})  (1-2^{-s})^{-1} \sum_{l=0}^{\infty}\frac{ \chi_D(2^l) }{2^{ls}} \notag \\
&=\tilde{P}_2(D, s)  (1- 2^{-2s}) (1-2^{-s})^{-1} \sum_{l=0}^{\infty}\frac{ \chi_D(2^l) a(D, 2^l)}{ 2^{ls}}. \notag
\end{align}

By the multiplicative property of $A(D, \cdot)$ and that of $\chi_D(\cdot) a(D, \cdot)$, 
 $$ \sum_{m>0} \frac{A(D, 4m)}{ m^{s}} =  \tilde{P}_2(D, s) \zeta(2s)^{-1} \zeta(s) \sum_{m>0} \frac{\chi_D(\hat{m}) a(D,m)}{ m^{s}} .$$

It remains to compute $\tilde{P}_2$. This is obtained by the next lemma.

\begin{lemma}
Let $D$ be an odd discriminant. With $\tilde{P}_2(D, s)$ defined in the last proposition, then
\begin{align*}
\tilde{P}_2(D, s)  = 2.
\end{align*}
\end{lemma} 
\begin{proof}
Write 
$$ \sum_{l =0} \frac{A(D,  2^{l+2})}{2^{ls} } = A(D, 4) +  \sum_{l =1} \frac{A(D,  2^{l+2})}{2^{ls} } ,$$
and note that 
$$A(D, 4)  = \begin{cases} 2 & D  \equiv1 \  \text{or} \ 5  \ (\text{mod} \  8), \\ 0 & \mathrm{otherwise}. \end{cases}$$
To simplify the second term, note that if $D$ is an odd integer and $m = 8,16$, or some higher power of $2$, then $D$ is a quadratic residue modulo $m$ if and only if $D \equiv1 (\text{mod} \  8)$, therefore for $l \geq 1$,
$$ A(D, 2^{l+2}) = \begin{cases} 4& D  \equiv1 \ (\text{mod} \  8) ,\\  0 & \text{otherwise} .\end{cases}$$
Also note that 
$$ \chi_D(2)  = \begin{cases} 1& D  \equiv1 \ (\text{mod} \  8) , \\ -1 & D \equiv 5 \ (\mathrm{mod} \ 8),\\  0 & \text{otherwise} .\end{cases}$$
Then direct computation gives the results. 
\end{proof}
This finishes the proof of the proposition. 
\end{proof}

\begin{theorem}
The Shintani zeta function $Z_{\mathrm{Shintani}}^{\mathrm{odd}}(s,w)$ and the quadratic $A_2$ Weyl group multiple Dirichlet series $Z_{A_2}(s, w)$ satisfy the equation:
\begin{align*}
Z_{\rm{Shintani}}^{\mathrm{odd}} (s, w)  =& 2 \zeta(2s)^{-1} \zeta(s) Z_{A_2}(s, w).  
\end{align*}
\end{theorem}

\section{Bhargava integer cubes and zeta functions} 
\subsection{Prehomogeneous vector space of a parabolic subgroup.}
Let  $V(\mathbb{Z})$ be the set of $2\times 2 \times 2$ integral matrices. For each element $A \in V(\bZ)$, there are three ways to form pairs of matrices by taking the opposite sides out of $6$ sides. Denote them by
\begin{align}
&M_A^1=\left(\begin{array}{ccc}a &b \\
c &d
\end{array}\right); N_A^1= \left( \begin{array}{ccc}e &f \\
g &h
\end{array}\right),
\notag \\
&M_A^2=\left(\begin{array}{ccc}a &e \\
c &g
\end{array}\right); N_A^2 =\left( \begin{array}{ccc}b &f \\
d &h
\end{array}\right),
\notag \\
&M_A^3=\left(\begin{array}{ccc}a &e \\
b &f
\end{array}\right);  N_A^3=\left( \begin{array}{ccc}c &g \\
d &h
\end{array}\right).
\notag 
\end{align}
For each pair $(M_A^i, N_A^i)$ we can associate to it a binary quadratic form by taking
$$Q_A^i(u,v) =- \det(M_A^i u - N_A^i v) .$$
Explicitly for $A$ as above,
\begin{align} -Q_A^1(u,v) &= u^2(ad-bc) +uv ( -ah+bg+cf-de) +v^2(eh-fg), \notag \\
-Q_A^2 (u,v)&= u^2(ag-ce)+uv(-ah-bg+cf+de) + v^2(bh-df), \notag \\
-Q_A^3(u,v)&= u^2(af-be) +uv(-ah+bg-cf+de) +v^2(ch-dg).\notag \end{align}

Following Bhargava \cite{bhargava}, we call $A$ projective if the associated binary quadratic forms are all primitive. The action of group $G(\bZ)=\SL_2( \bZ) \times \SL_2( \bZ) \times \SL_2(\bZ)$ on $V(\bZ)$ is defined by taking the $g_i$ in $\left( g_1, g_2, g_3 \right)$ acts on the matrix pair $(M_A^i, N_A^i)$. It is easy to check that the actions of the three components commute with each other, thereby giving an action of the product group. For example, if $g_1 = \left( \begin{matrix} g_{11} & g_{12} \\
g_{21} & g_{22} \end{matrix} \right) $, then it acts on the pair $(M_A^1, N_A^1)$ by 
$$\left( \begin{matrix} g_{11} & g_{12} \\
g_{21} & g_{22} \end{matrix} \right) \cdot \left( \begin{matrix} M_A^1 \\ N_A^1 \end{matrix} \right).$$

The action extends to the complex group $G(\bC)= \GL_2(\bC) \times \GL_2(\bC) \times \GL_2(\bC)$ on the complexified vector space $V(\bC)$. Now we consider the Borel subgroup $B_2( \bC) \subset \GL_2(\bC)$ consisting of the lower-triangular matrices. The action of the subgroup $B_2(\bC) \times B_2(\bC) \times \GL_2(\bC)$ is induced from full group action and has three relative invariants, explicitly for $A \in V(\bC)$, given as follows
\begin{align}
D(A) & =\mathrm{disc}(A)  =(-ah+bg+cf-de)^2 - 4 (ad-bc)(eh-fg) , \notag \\
m(A) & =-\det(M_A^1) =-(ad-bc) ,\notag \\
n(A) &=-\det(M_A^2) =-(ag-ce). \notag 
\end{align}
The corresponding rational character are
\begin{align*}
\chi_1(g) &= \det(b_1)^2 \det(b_2)^2 \det(g_3)^2,\\
\chi_2(g) &= r_1^2 \det(b_2) \det(g_3),\\
\chi_3(g) &= \det(b_1) r_2^2 \det(g_3),
\end{align*}
for $g = (b_1, b_2, g_3) = \left( \left( \begin{matrix}  
r_1 & 0 \\
u_1 & s_1
\end{matrix} \right),  \left( \begin{matrix}  
r_2 & 0 \\
u_2 & s_2
\end{matrix} \right), g_3 \right)$.
\\Furthermore, one can easily show that
\begin{proposition}
The pair $\left(B_2(\bC) \times B_2(\bC) \times \GL_2(\bC) , V(\bC)\right)$ is a prehomogeneous vector space.
\end{proposition}
The immediate corollary is that $\left(G(\bC), V(\bC)\right)$ is again a PVS which is called the $D_4$ type studied in \cite{wrightyukie}.  

Let $B_2(\bZ)$ be the Borel subgroup of $B_2(\bC)$ in $\SL_2(\bZ)$ with positive diagonal elements. The Shintani zeta function associated to the prehomogeneous vector space $(B_2(\bC) \times B_2(\bC) \times \GL_2(\bC) , V(\bC))$ is defined to be

\begin{equation*} \tag{3}
Z_{\rm{Shintani}}(s_1, s_2, w) = \sum_{A \in \left( B_2(\bZ) \times B_2(\bZ) \times \SL_2(\bZ) \right) \backslash V^{ss}(\bZ) } \frac{1}{|\mathrm{disc}(A)|^w |\det(M_A^1)|^{s_1} |\det(M_A^2)|^{s_2}} ,\notag 
\end{equation*}
where $V^{ss}(\bZ)$ is the subset of semi-stable points of $V(\bZ)$ consisting of those orbits on which none of the three relative invariants vanishes. Denote by $\mathrm{Stab}(A)$ the stabilizer group in $ B_2(\bZ) \times B_2(\bZ) \times \SL_2(\bZ)$ of the cube $A$. We need to show that the order $|\mathrm{Stab}(A)|$ is finite. 
\begin{proposition}
For any $A \in V^{ss}(\bZ)$, the stabilizer group $\mathrm{Stab}(A)$ in $B_2(\bZ) \times B_2(\bZ) \times \SL_2(\bZ)$ is:
$$\{ (I_2, I_2, I_2)\},$$
where $I_2$ is the $2\times 2$ identity matrix. 
Therefore
$$|\mathrm{Stab}(A)| =1$$
for any   $A \in V^{ss}(\bZ)$.
\end{proposition}
\begin{proof}
For a given $2\times 2\times 2$ integer cube $A  \in V^{ss}(\bZ)$, we write 
\begin{align}
Q^1_A(u, v) &= mu^2+ x uv + s v^2 , \notag \\
Q^2_A(u, v) &= n u^2 + y uv + t v^2 \notag
\end{align}
to be first two binary quadratic forms associated to it. Under the action of $ B_2(\bZ) \times B_2(\bZ) \times \SL_2(\bZ)$ we can change $A$ to another cube satisfying 
$$ a(A) =0 \  \mathrm{and} \  0 \leq x \leq 2 |m|-1 \ \mathrm{and} \  0 \leq y \leq 2 |n|-1.$$
For such $A$, note that the entries $c(A), d(A), e(A) \neq 0$. Therefore the stabilizer group in $ B_2(\bZ) \times B_2(\bZ) \times \SL_2(\bZ)$ must have the form
$$
\left( \left( \begin{matrix} 1 &  0 \\
0 & 1 \end{matrix} \right),
\left( \begin{matrix}  1 &  0 \\
0 & 1 \end{matrix} \right),
\left( \begin{matrix}  \pm 1 &  0 \\
0 & \pm 1 \end{matrix} \right) \right).
$$
Now it becomes obvious that the diagonal elements in the last matrix have to be both positive.
\end{proof}

We can rewrite the Shintani zeta function as 
\begin{equation*} \tag{4}
Z_{\rm{Shintani}}(s_1, s_2, w) = \sum_{D \neq 0} \frac{1}{|D|^w} \sum_{m,n >0 } \frac{B(D, m,n)}{ m^{s_1} n^{s_2}} ,\notag 
\end{equation*}
where 
\begin{align}B(D, m, n)  &  = \# \{ A \in   V^{ss}(\bZ) / \sim : \mathrm{disc}(A) =D, |\det(M_A^1)| =m,| \det(M_A^2)| =n \}  \notag \\
& = 4 \cdot \# \{ A \in   V^{ss}(\bZ) / \sim : \mathrm{disc}(A) =D, \det(M_A^1) =m, \det(M_A^2) =n \}. \notag 
\end{align}

\subsection{Reduction theory}
In this section, we will consider the geometry of integer orbits of the PHVS under the parabolic group action. This will help to reduce the sum over semi-stable orbits of $2\times 2\times 2$ integer cubes in the Shintani zeta function $Z_{\mathrm{Shintani}}(s_1, s_2, w)$ to a sum over the tuples of integers $(D, m, n)$ satisfying certain relations. First we need a lemma which establishes the existence of a $2\times 2\times 2 $ integer cube with the required arithmetic invariants. 

\begin{lemma}
Let $D$, $m$ and $n$ be non-zero integers. For each solution $(x, y)$ to the congruence equations 
\begin{align} 
x^2 &\equiv D  \ ( \mathrm{mod} \ 4m)  \ \text{for} \ 0\leq x \leq 2 |m| -1, \notag \\
y^2  &\equiv D \ ( \mathrm{mod} \ 4n\ )  \ \text{for} \ 0\leq y \leq 2 |n| -1, \notag 
\end{align}
there exists a $2\times 2\times2$ integer cube $A$ such that 
\begin{align}
\mathrm{disc}(A) &= D , \notag \\
Q^1_A(u,v) & = mu^2+ x uv+ sv^2, \notag \\
Q^2_A(u,v) &=n u^2+ y uv + t v^2. \notag
\end{align}
Moreover, the required $2 \times 2 \times 2$ integer cube $A$ can be chosen such that in the top side $M_A^3$ of $A$
\begin{align*}
a=0 \ \mathrm{and} \ g.c.d.(b, e, f) =1. 
\end{align*}
\end{lemma}
\begin{proof}
If there are solutions to the congruence equations, we have
$$D  =x^2 - 4 m s  = y^2 - 4 n t$$
for some integers $s$ and $t$. It implies that $D$ is congruent to $0$ or $1 \ (\mathrm{mod} \ 4)$, and the integers $x, y$ have the same parity. Take 
$$ c= |g.c.d.(m, n, \frac{x+y}{2})|,$$
from 
\begin{equation*}\tag{5} \frac{x-y}{2} \cdot \frac{x+y}{2} = ms - nt, \end{equation*}
it follows that 
\begin{equation*}\tag{6} g.c.d.(\frac{m}{c}, \frac{n}{c})\  | \ \frac{x-y}{2}.\end{equation*}
Set
$$b = \frac{m}{c}, \ e = \frac{n}{c}, \ \text{and} \ f=-\frac{x+y}{2 c}, $$
then 
$$g.c.d.(b,e,f) =1,$$
and $(5)$ can be written as
\begin{equation*} \tag{7}\frac{x-y}{2} \cdot (-f) = bs- et.\end{equation*}

We claim that there is an integer $h$, such that 
\begin{align}
s+ eh & \equiv 0 \ (\text{mod}\ f), \notag \\
t + bh & \equiv 0 \ (\text{mod}\ f). \notag
\end{align}
This can be proved as follows: First if $f=0$, then $bs = et$. As in this case $g.c.d.(b,e) =1$, we conclude that there exists such an integer $h$ such that $ s= -eh$ and $t = -bh$. If $f \neq 0$, for any prime divisor $p$ of $f$, we have 
$$p\ | \ bs- et \  \text{and} \ g.c.d.(b,e,f)=1,$$
it follows that there is a unique solution ($\text{mod} \ p$) to the congruences 
\begin{align}
s+ eh & \equiv 0 \ (\text{mod}\ p), \notag \\
t + bh & \equiv 0 \ (\text{mod}\ p). \notag
\end{align}
Using the Chinese remainder theorem, we conclude that there are integers $d$ and $g$ such that
\begin{align}
s = gf - eh, \notag \\
t = df - bh. \notag
\end{align}
It follows that 
$$ bs- et = (bg- de) \cdot f,$$  
combined with $(7)$, we have
$$\frac{x-y}{2}= de- bg.$$
If in the case of $f=0$, from $(6)$, we know that there always exist integers $d$ and $g$ such that the above equation holds. 

Now we define a $2 \times 2\times 2$ integer cube $A$ by 
$$
(M_A^1, M_A^2)= \left(\left( \begin{array}{cc} a & b \notag \\ 0 & d \notag \end{array} \right) , 
 \left( \begin{array}{cc} e & f \notag \\ g & h \notag \end{array} \right) \right). 
$$
Then the two associated binary quadratic forms are
\begin{align}
Q^1_A(u, v)&= bc u^2 -( bg+cf-de)uv -(eh-fg) v^2 = m u^2+ x uv + s v^2, \notag \\
Q^2_A(u,v) &= ce u^2 - ( - bg+cf+ de) uv-(bh- df\ ) v^2 = nu^2+ y uv+ t v^2. \notag 
\end{align}
So $A$ is the cube required. In particular, we have shown that $g.c.d.(b,e,f) =1$. 
\end{proof}

We next want to show that under the assumption that $D$ is square-free, the data $(D, m, n, x, y)$ uniquely determines a $B_2(\bZ) \times B_2(\bZ) \times \SL_2(\bZ)$-orbit.  

\begin{proposition}
Let $D$ be a non-zero square-free integer and $m$, $n$ be non-zero integers. Each solution to the congruence equations
\begin {align}
x^2 &\equiv D  \ ( \mathrm{mod} \ 4m)  \ \text{for} \ 0\leq x \leq 2 |m| -1, \notag \\
y^2  &\equiv D\ ( \mathrm{mod} \ 4n\ )  \ \text{for} \ 0\leq y \leq 2 |n| -1, \notag 
\end{align}
determines uniquely an integral orbit represented by $A$ such that 
\begin{align}
\mathrm{disc}(A) &= D , \notag \\
Q^1_A(u,v) &= mu^2+ x uv+ s v^2, \notag \\
Q^2_A(u,v) &=nu^2 + yuv + t y^2. \notag
\end{align}
\end{proposition}
\begin{proof}
From the last lemma, we know that there always exists such a $2 \times 2 \times 2$ integral matrix $A$ satisfying the congruence equations. We want to show that the integral orbit it represent in the 
PVS $\left(B_2(\bC) \times B_2(\bC) \times \GL_2(\bC), V(\bC)\right)$
is uniquely determined by the data $(d, m, n, x, y)$. 

Denote $A$ by the pair of matrices 
$$A = \left( \left( \begin{matrix}
a & b \\
c & d 
\end{matrix} \right ) ,  \left( \begin{matrix}
e & f \\
g & h 
\end{matrix} \right )  \right). $$
Under the action of the subgroup $1\times 1 \times \SL_2(\bZ)$, $a(A)$ can be made equal to $0$, so we assume it. We then show that the $(c, b, e, f)$ is uniquely determined up to the sign of $c$.
Then we have equations
\begin{align}
-bg-cf + de &= x, \notag \\
bg-cf-de&= y ,\notag \\
bc &= m ,\notag \\
ce &= n. \notag
\end{align}
after adding the first two equations, 
\begin{align}
cf &= -(x+y) /2, \notag\\
bc &= m, \notag \\
ce &= n. \notag
\end{align}
Therefore we have 
$$ c \times g.c.d.(b, e,f) = g.c.d.(m, n, (x+y)/2).$$
As 
$$D = (bg-ah -ed)^2+4bc (eh-fg), $$
and $D$ is square-free, it implies that $g.c.d.(b, e,f) =1$. Therefore $c = |g.c.d.(m,n,(x+y)/2)|$ as we can transform $c$ to be positive. 

We next show that for two such $2\times 2\times 2$ integral matrices with $a =0$ they are equivalent in the sense one is transformed to another by the action of $1 \times 1 \times B_2(\bZ)$.
To prove this statement we again require $D$ to be square-free.  

As both have the same associated binary quadratic forms:
\begin{align}
Q^1_A(u,v) &= m u^2 + x uv + sv^2 \notag  =bc u^2 -(bg+ cf-ed)uv-(eh-fg)v^2 ,\notag \\
Q^2_A(u, v) &= n u^2 + y uv + tv^2 \notag = ce u^2 -(-bg+cf+de) uv-(bh-df) v^2 ,\notag
\end{align}
we can set up the following equations by comparing the coefficients:
\begin{align}
x(A_1) = x(A_2): & bg_1+cf - e d_1= b g_2 +cf - e d_2 , \notag \\
y(A_1) = y(A_2): & ed_1 +cf - bg_1 = e d_2+cf - b g_2, \notag \\
s(A_1) =s(A_2) : & e h_1-f g_1 = e h_2- fg_2 , \notag \\
t(A_1) = t(A_2):  &b h_1- d_1f = b h_2 - d_2f .\notag 
\end{align}
Subtracting the right side from the left side on each equation, we have:
\begin{align}
b \Delta (g) - e \Delta( d) =0, \notag \\
e \Delta (h) -f \Delta (g) = 0, \notag \\
b \Delta (h)- f \Delta (d)  = 0, \notag
\end{align}
where we use notation $\Delta(g)= g_1-g_2$. Then we have solutions to the above equation system:
$$ \Delta(d); \Delta(g) = \Delta(d) e/b;   \Delta(h) = \Delta(d)f /b.$$
Under the assumption of $D$ square-free, we have shown that $g.c.d.(b, e, f)$ $=1$. Therefore $b|\Delta(d)$, i.e., the solution has form 
$$ \Delta(d)=bk; \Delta(g) = e k;   \Delta(h) = f k.$$
This implies that we can make $A_1$ equivalent to $A_2$ by the action of $1 \times  1  \times B_2(\bZ)$. 
\end{proof}

Now we turn to the general case without assuming square-free discriminant. For this, we denote by $B(D, m,n)$ the $4$ times the number of integral orbits $[A]$ in $(B_2(\bC) \times B_2(\bC) \times \GL_2(\bC), V(\bC))$ with 
$\mathrm{disc(A)}=D, Q^1_A(1, 0) =m, Q^1_A(1, 0) =n$. Then

\begin{proposition}
Let $m$ and $n$ be non-zero integers, and $D= D_0 D_1^2$ where $D_0$ is square-free. We have
\begin{equation*} \tag{8} B(D,m,n) = \sum_{d|D_1} b(\frac{D}{d^2}, \frac{m}{d}, \frac{n}{d}),\end{equation*}
where
\begin{align*}
b(\frac{D}{d^2}, \frac{m}{d}, \frac{n}{d})= \begin{cases}  d \cdot A(\frac{D}{d^2}, \frac{4m}{d}) \cdot
A(\frac{D}{d^2}, \frac{4n}{d}) & \mathrm{if} \ d \ \mathrm{ divides} \ g.c.d.(D_1, m,n) , \\
0 &  \mathrm{otherwise}. 
\end{cases}
\end{align*}
\end{proposition}

\begin{proof}
By the existence lemma, there is a $2 \times 2 \times 2$ integral matrix $A$ satisfying
\begin{align*}
a =0 \ \mathrm{and} \ g.c.d.(b,e,f) =1,
\end{align*} 
and such that 
\begin{align*}
\mathrm{disc}(A) = D, \ Q^1_A(1, 0) =m \ \mathrm{and} \ Q^1_A(1, 0) =n. 
\end{align*}
The condition $g.c.d.(b,e,f) =1$ implies that the matrix $A$ determines a unique integral orbit.  

If $d| g.c.d.(D_1, m,n)$, we write $m'= \frac{m}{d}$, $n'= \frac{n}{d}$. Applying the existence lemma again to the non-zero integers $\frac{D}{d^2}$, $\frac{m}{d}$, and $\frac{n}{d}$, we can 
find an $2 \times 2 \times 2$ integral matrix $A'$ satisfying 
\begin{align*}
a' =0 \ \mathrm{and} \ g.c.d.(b',e',f') =1,
\end{align*} 
and such that 
\begin{align*}
\mathrm{disc}(A') = \frac{D}{d^2}, \ Q^1_A(1, 0) = \frac{m}{d} \ \mathrm{and} \ Q^2_A(1, 0) =\frac{n}{d}. 
\end{align*}
The condition $g.c.d.(b', e', f')=1$ implies that $A'$ determines a unique integral orbit with discriminant $D/d^2$.  

For each $g$ in $\{1 \times 1 \times \left( \begin{array}{cc }1 & 0 \\
j & d\end{array}\right): 0 \leq j \leq d-1\} $, the integral cube $g \cdot A'$ represents an integral orbit with discriminant $D$. 
\end{proof}

As an immediately corollary, we have
\begin{corollary}
If $D$ is a fundamental discriminant, then
$$B(D, m,n)=  A(D, 4m) A(D, 4n).$$
\end{corollary}

Finally we obtain the explicit formula for the Shintani zeta function associated to the PVS of $2\times 2 \times 2$ cubes. 
\begin{theorem}
The Shintani zeta function $Z_{\rm{Shintani}}(s_1, s_2, w)$ can be expressed as  
\begin{align}
& \sum_{A \in \left( B_2(\bZ) \times B_2(\bZ) \times \SL_2(\bZ) \right) \backslash V^{ss}(\bZ) } \frac{1}{|\mathrm{disc}(A)|^w |\det(A^F)|^{s_1} |\det(A^L)|^{s_2}} \notag \\
& =  \sum_{D=D_0D_1^2} \frac{1}{|D|^w} \sum_{m,n > 0} \frac{\sum_{\substack{d|D_1\\ d|m, d|n}}d\cdot A(\frac{D}{d^2}, \frac{4m}{d}) \cdot A(\frac{D}{d^2}, \frac{4n}{d})}{m^{s_1}n^{s_2}}. \notag 
\end{align}
\end{theorem}

In particular, if $D$ is an odd integer, then $(8)$ becomes
\begin{equation*}\tag{9}B(D,m,n) =\sum_{\substack{d|D_1}}d\cdot  A(\frac{D}{d^2}, \frac{4m}{d}) \cdot A(\frac{D}{d^2}, \frac{4n}{d}),\end{equation*}
as when $d$ is an odd integer it is a divisor of $m$ if and only if it is a divisor of $4m$. So we have the following 

\begin{corollary}
The partial Shintani zeta function $Z^{\mathrm{odd}}_{\rm{Shintani}}(s_1, s_2, w)$ defined by
\begin{align*}
 Z^{\mathrm{odd}}_{\rm{Shintani}}(s_1, s_2, w)=\sum_{\substack{ A \in \left( B_2(\bZ) \times B_2(\bZ) \times \SL_2(\bZ) \right) \backslash V^{ss}(\bZ)\\ \mathrm{dics}(A) \ \mathrm{odd} }} \frac{1}{|\mathrm{disc}(A)|^w |\det(A^F)|^{s_1} |\det(A^L)|^{s_2}} 
\end{align*}
can be written as
\begin{align*}
 \sum_{D=D_0D_1^2} \frac{1}{|D|^w} \sum_{m,n > 0} \frac{\sum_{d|D_1}d\cdot  A(\frac{D}{d^2}, \frac{4m}{d}) \cdot A(\frac{D}{d^2}, \frac{4n}{d})}{m^{s_1}n^{s_2}}. 
\end{align*}
\end{corollary}

\section{$A_3$ Weyl Group Multiple Dirichlet Series}
In this section, we will relate the Shintani zeta function $Z_{\mathrm{Shintani}}(s_1, s_2, w) $ to the quadratic $A_3$-Weyl group multiple Dirichlet series. The idea is first to construct a multiple Dirichlet series $Z_{\rm{WMDS}}(s_1, s_2, w) $ and then show its relation to the Shintani zeta function of PVS of $2 \times 2 \times 2$ cubes, using the results we did for the relation between $Z_{A_2}(s, w)$ and $Z_{\mathrm{Shintani}}(s, w; B_2)$. Finally we show the multiple Dirichlet series $Z_{\mathrm{WMDS}}(s_1, s_2, w)$ is the desired $A_3$ Weyl group multiple Dirichlet series by computing the generating function of its $p$-parts. 

In the $A_2$-WMDS 
\begin{align*}
Z_{A_2} (s, w) = \sum_{\substack{m >0, \\  D \ \mathrm{odd}\ \mathrm{discriminant}}} \frac{\chi_D(\hat{m}) a(D, m)}{ m^s |D|^w},
\end{align*} 
we let $\tilde{A}(D, m) = \chi_D(\hat{m}) a(D, m)$. Write $D= D_0 D_1^2$ where $D_0$ stands for the square-free part of $D$.  Define
\begin{align*}
Z_{\mathrm{WMDS}}(s_1,s_2, w) 
= \sum_{ D \ \mathrm{odd} \ \mathrm{discriminant}} \frac{1}{|D|^w} \sum_{m,n>0} \frac{\chi_D(\hat{m}) \chi_D(\hat{n})}{m^{s_1}n^{s_2}}a(D,m,n) ,
\end{align*}
where
\begin{align*}
a(D, m,n) = \sum_{\substack{d|D_1\\ d|m,\ d|n}} d\cdot  a(\frac{D}{d^2}, \frac{m}{d}) \cdot a(\frac{D}{d^2}, \frac{n}{d}). 
\end{align*}

\begin{lemma}
The multiple Dirichlet series $Z_{\mathrm{WMDS}}(s_1, s_2, w)$ can be expressed as 
\begin{align*}
Z_{\mathrm{WMDS}}(s_1,s_2, w) 
&= \sum_{ D\ \mathrm{odd} \ \mathrm{discriminant}} \frac{1}{|D|^w} \sum_{m,n>0} \frac{\chi_D(\hat{m}) \chi_D(\hat{n})}{m^{s_1}n^{s_2}}a(D,m,n) \\
& =   \sum_{\substack{D= D_0 D_1^2 \\ D\ \mathrm{odd} \ \mathrm{discriminant} }} \frac{1}{|D|^w}\frac{\sum_{\substack{d|D_1\\ d|m, d|n}}d \cdot \tilde{A}(\frac{D}{d^2},\frac{m}{d}) \cdot \tilde{A}
(\frac{D}{d^2}, \frac{n}{d})}{m^{s_1}n^{s_2}} . 
\end{align*}
\end{lemma}
\begin{proof}
This follows from the fact that the quadratic character $\chi_{D}(\cdot)$ is the same as $\chi_{D/d^2}(\cdot)$ by definition, and the factor $m$ is prime to $D$ if and only if the factor $\frac{m}{d}$ is prime to $\frac{D}{d^2}$. Therefore $\chi_D(\hat{m}) \chi_D(\hat{n})$ is a common factor for fixed integers $D$, $m$ and $n$. 
\end{proof}

As we did in the Proposition 2.5, we will show the relation between the inner sum of $Z^{\mathrm{odd}}_{\mathrm{Shintani}}(s_1, s_2, w)$ and $Z_{\rm{WMDS}}(s_1, s_2, w)$. 
\begin{proposition}
Let $D$ be an odd integer. The inner sum of the Shintani zeta function $Z_{\mathrm{Shintani}}^{\mathrm{odd}} (s_1, s_2, w)$ can be expressed by 
$$\sum_{m, n > 0} \frac{B(D, m, n)}{m^{s_1} n^{s_2}}=4 \frac{\zeta(s_1)}{\zeta(2 s_1)} \frac{\zeta(s_2)}{\zeta(2 s_2)}  \sum_{m,n>0} \frac{\chi_D(\hat{m}) \chi_D(\hat{n})}{m^{s_1}n^{s_2}}a(D,m,n) .$$
\end{proposition}
\begin{proof}
Recall that in the proof of Proposition 2.5, we have shown that 
\begin{align*}
\sum_{m>0} \frac{A(D, 4m)}{m^{s} }  = 2 \frac{\zeta(s)}{\zeta(2 s)}  \sum_{m>0} \frac{\chi_D(\hat{m})a(D,m)}{ m^{s}} . 
\end{align*}
Therefore, 
\begin{align*}
\sum_{m, n> 0} \frac{A(D, 4m) A(D, 4n)}{m^{s_1} n^{s_2}} & = 4  \frac{\zeta(s_1)}{\zeta(2 s_1)} \frac{\zeta(s_2)}{\zeta(2 s_2)}    \\
& \cdot \sum_{m, n>0} \frac{\chi_D(\hat{m})\chi_D(\hat{n})}{ m^{s_1} n^{s_2}}  a(D, m) a(D, n). 
\end{align*}
In particular for any $d^2|D$, replace $D$ by $\frac{D}{d^2}$, $m$ by $\frac{m}{d}$ and $n$ by $\frac{n}{d}$, we have
\begin{align*}
\sum_{m, n> 0} \frac{A(\frac{D}{d^2}, \frac{4m}{d}) A(\frac{D}{d^2}, \frac{4n}{d})}{m^{s_1} n^{s_2}} & = 4 \frac{\zeta(s_1)}{\zeta(2 s_1)}  \frac{\zeta(s_2)}{\zeta(2 s_2)}   \\
& \cdot \sum_{m, n>0} \frac{\chi_D(\hat{m})\chi_D(\hat{n})}{ m^{s_1} n^{s_2}}  a(\frac{D}{d^2}, \frac{m}{d}) a(\frac{D}{d^2}, \frac{n}{d}). 
\end{align*}
Finally, taking the sum over all $d$ with $d^2|D$, we have
\begin{align*}
\sum_{m,n > 0} \sum_{d^2|D} \frac{ d \cdot A(\frac{D}{d^2}, \frac{4m}{d}) \cdot A(\frac{D}{d^2}, \frac{4n}{d})}{m^{s_1}n^{s_2}} &= 4 \frac{\zeta(s_1)}{\zeta(2 s_1)}  \frac{\zeta(s_2)}{\zeta(2 s_2)}  \\
& \cdot \sum_{m,n>0} \frac{\chi_D(\hat{m}) \chi_D(\hat{n})}{m^{s_1}n^{s_2}}a(D,m,n) . 
\end{align*}
\end{proof}

Now we are able to give an explicit relation between two multiple zeta functions $Z_{\mathrm{Shintani}}(s_1, s_2, w)$ and $Z_{\rm{WMDS}}(s_1, s_2, w)$. We further relate the Shintani zeta function $Z_{\mathrm{Shintani}}(s_1, s_2, w)$ to a Weyl group multiple Dirichlet series by showing that $Z_{\rm{WMDS}}(s_1, s_2, w)$ is a quadratic $A_3$-WMDS. 

\begin{theorem}
The Shintani zeta function of $\mathrm{PVS}$ of $2\times 2 \times 2$ cubes can be related to the multiple Dirichlet series $Z_{\mathrm{WMDS}}(s_1, s_2, w)$ by 
\begin{align*}
Z^{\mathrm{odd}}_{\rm{Shintani}}(s_1, s_2, w) = & 4 
\frac{\zeta(s_1)}{\zeta(2 s_1)} \frac{\zeta(s_2)}{\zeta(2 s_2)}   Z_{\rm{WMDS}} (s_1, s_2, w).    
\end{align*}
\end{theorem}

\begin{theorem}
$Z_{\rm{WMDS}}(s_1,s_2,w)$ is a quadratic $A3$ Weyl group multiple Dirichlet series.
\end{theorem}
\begin{proof} Consider the $p$-parts of our $a(D,m,n)$ defined by
$$ a_{klt}(p)=a(p^k, p^l, p^t).$$
Explicit from its definition,
$$ a_{klt}(p) = a(p^k, p^l) a(p^k,p^t) + p a(p^{k-2},p^{l-1}) a(p^{k-2}, p^{t-1})+ \cdots .$$
In \cite{chintagunnells} and \cite{chintagunnellsjams}, the authors developed a systematical way to construct the Weyl group multiple Dirichlet series. The idea is to construct a rational function invariant under the Weyl group action. In the case of root system of $A_3$ type, the $p$-parts of the rational function is given by 
$$f_{A_3}(x,y,z)= \frac{(1 - xy - yz + xyz + pxy^2z - px^2y^2z - pxy^2z^2 + 
   px^2 y^3 z^2)}{(1 - x)(1 - y)(1 - z)(1 - py^2z^2) (1 - 
     py^2x^2) (1 - p^2x^2y^2z^2)}.$$
Write the expansion
$$f_{A_3}(x,y,z)= \sum b_{klt}(p) x^l y^k z^t,$$
then we can compare our $\{a_{klt}\}$ with $\{b_{klt}\}$. They coincide with each other as follows: we know that for $|x|, |y|, |z| < 1/p$ there is
$$f_{A_3}(x,y,z) = \frac{1}{1-pxy^2z} \int f_{A_2}(x, t)f_{A_2}(yt^{-1},z)\frac{dt}{t},$$
where 
$$f_{A_2}(x_1,x_2) = \sum_{k,l \geq 0} a(p^k,p^l) x_1^k x_2^l$$
and the integral is taken over the circle $|t| =1/p$ (\cite[Example 3.7]{chintagunnells}). Substituting the above expansion of $f_{A_2}$ into $f_{A_3}$, note that $a(p^k, p^l) = a(p^l, p^k)$, we have
\begin{align}
f_{A_3}(x,y,z) & = \frac{1}{1-pxy^2 z} \sum_{k,l,t \geq 0} a(p^k,p^l) a(p^k,p^t) x^l y^k z^t  \notag \\
 &= \sum_{s=0}^{\infty} p^s x^s y^{2s} z^s  \sum_{k,l,t \geq 0} a(p^k,p^l)a(p^k,p^t) x^ly^kz^t \notag \\
& =\sum_{k,l,t \geq 0} ( a(p^k, p^l) a(p^k,p^t) + p a(p^{k-2},p^{l-1}) a(p^{k-2}, p^{t-1})+ \cdots ) x^ly^kz^t \notag \\
& = \sum_{k.l,t \geq 0} a(p^k,p^l, p^t) x^l y^k z^t. \notag
\end{align}
Therefore, $a_{klt} = b_{klt}$. 
\end{proof}

\section{Moduli parameterizes ideals of a quadratic ring}
We first recall the classical results in the theory of binary quadratic forms \cite{cox}. Given a primitive integral binary quadratic form
$$Q(x, y) = ax^2 + bx y +  cy^2$$
with discriminant $D= b^2 -4 ac$ such that $K =\bQ ( \sqrt{D})$ a quadratic field, there is a canonical way to associate it an integral ideal $I$ in the quadratic order $R=R(D)$. 
Let $\tau $ be one of the two roots of the quadratic function $Q(x, 1) =0$, then  
\begin{align}
R = \langle 1, a \tau \rangle \notag \ \mathrm{and} \ 
I = \langle a, a \tau \rangle .\notag 
\end{align}
If $f$ is the conductor of the quadratic order $R$, then
we can express $a\tau$ as:
$$a \tau = \frac{-b \mp f d_K }{2} \pm f w_K, $$
where $ w_ K = \frac{d_K+ \sqrt{d_K}} {2}$, and $d_K$ is the fundamental discriminant of the quadratic field $K$. It follows that $R$ has a $\bZ$-basis $[1, f w_K]$, which only depends on the discriminant $D=f^2 d_K$ of $R$. Notice that $I$ is the ideal of $R$ satisfying $R/I \cong \mathrm{N}(I) \bZ$, where the norm $\mathrm{N}(I) =|a|$. 

In order to extend the above construction to an arbitrary integral binary quadratic form
\begin{equation*} \tag{10}Q(x, y) = ax^2 + bx y +  cy^2 \end{equation*}
with discriminant $D= b^2 - 4 ac \neq 0$, we need to first recall the definition of oriented quadratic ring introduced in the paper \cite{bhargava}. A quadratic ring is the commutative ring with unity whose underlying additive group is $\bZ^2$. There is a unique automorphism for a quadratic ring $R$. With the automorphism, we can define the trace of an element $x \in R$ by taking $\mathrm{Tr}(x) = x+ x'$, where $x'$ denotes the image of $x$ under the automorphism. Alternatively, the trace function $\mathrm{Tr} : R \to \bZ$ is defined as the trace of the endomorphism $R \xrightarrow{ \times \alpha} R$. We also define the norm of an element $x \in R$ by taking $\mathrm{N}(x) = x \cdot x'$. The discriminant $\mathrm{disc}(R)$ of $R$ is defined to be the determinant $\det( \mathrm{Tr}(\alpha_i \alpha_j))$ where $\{ \alpha_i\}$ is any $\bZ$-basis of $R$. As the $\bZ$-basis of any quadratic $R$ has the form $[1, \tau]$, where $\tau$ satisfies the equation $\tau^2 + r \tau + s = 0$, the discriminant of $R$ is given explicitly by $\mathrm{disc}(R) = r^2 - 4 s$. Conversely, given any integer $D  \equiv 0$ or $1 \  ( \rm{mod} \ 4)$, there exists a unique quadratic ring $R(D)$ with discriminant $D$. Canonically, $R(D)$ has a $\bZ$-basis $[1, \tau_D]$, where $\tau_D$ is determined by 
\begin{equation*}\tag{11}\tau_D^2 = \frac{D} {4} \ \mathrm{or} \ \tau_D^2 = \frac{D-1}{4} +\tau_D,\end{equation*}
in accordance to whether $D  \equiv 0 \ ( \rm{mod} \ 4)$ or $D  \equiv 1 \ ( \rm{mod} \ 4)$. We call $R$ non-degenerate if $\mathrm{disc}(R) \neq 0$. From now on, we only consider the case of non-degenerate quadratic rings. 

For an integer $D \neq 0$, the quadratic ring $R(D)$ has a unique non-trivial automorphism. The quadratic ring $R(D)$ is oriented if we specify the choice of $\tau_{D}$. For an oriented quadratic ring $R(D)$, the specific choice of $\tau$ in any $\bZ$-basis $[1, \tau]$ is made such that the change-of-basis matrix from the basis $[1, \tau]$ to the canonical basis $[1, \tau_D]$ has positive determinant. We call such a basis $[1, \tau]$ positively oriented.  For the rest of the section, we always assume that the quadratic ring $R(D)$ is oriented with each $\bZ$-basis $[1, \tau]$ positively oriented. 

Finally, for a quadratic ring $R$ with non-zero discriminant $D$, we define for it the narrow class group $\mathrm{Cl}^{+}(R)$, the group of oriented ideal classes. Recall that an oriented ideal is the pair $(I, \epsilon)$, where $I$ is a (fractional) ideal of $R$ in $K(R) = R \otimes \bQ$, and $ \epsilon = \pm 1$ gives the orientation of the ideal $I$. For an element $ k\in K(R)$, the product $ k \cdot (I, \epsilon)$ is defined to be the oriented ideal $(k I, \mathrm{sgn}\left(\mathrm{N}(k)\right) \epsilon)$. Two oriented ideal $(I_1, \epsilon_1)$ and $(I_2, \epsilon_2)$ belong to the same oriented ideal class if they satisfy $(I_1, \epsilon_1) = k \cdot(I_2, \epsilon_2)$. We will suppress $\epsilon$ for the rest of the section and assume $I$ always oriented. For an oriented ideal $I \subset R$, the unoriented norm of $I$ is defined to be $\mathrm{N}(I) = |R /I|$; while the oriented norm of $I$ is denoted by $\epsilon\cdot  \mathrm{N}(I)$.

Now for the binary quadratic form $(10)$, the oriented quadratic ring $R$ is defined to be 
$$R = \langle 1, \tau\rangle ,$$
where the choice of $\tau$ is specified and it satisfies 
$$ \tau^2 + b \tau + ac =0.$$
Further, the oriented ideal $I$ is defined to be
$$I = \langle a, \tau \rangle$$
with the orientation given by the ordered basis $[a, \tau]$. 
It is easy to see that $I$ is the ideal contained in $R$ with the norm $\mathrm{N}(I) = |a|$. If the binary quadratic form is primitive, i.e., $g.c.d.(a,b,c) =1$, then the ideal $I$ is proper, which means it has an inverse in the quadratic algebra $K(R)$.

Let $B_2(\mathbb{Z}) \subset \SL_2(\bZ)$ be the subgroup of lower-triangular integer matrices with positive diagonal elements. Then we will show the following result which says the set of pairs $(R, I)$ with oriented ideal $I$ with cyclic quotient in $R$ can be parameterized by the integer orbits of the PVS of binary quadratic forms acted on by the Borel subgroup $B_2(\bC)$. 

\begin{proposition}
The natural map 
\begin{align*}
 &B_2(\bZ)  \backslash  \{Q(u, v) = au^2 + b uv + cv^2: b^2 -4ac \neq 0, a \neq 0 \}  \\
& \to    \mathrm{Iso}  \backslash \{(R, I): R/I \cong \mathrm{N}(I) \bZ \}
\end{align*}
defined above is a bijection. The isomorphism $f$ from the pair $(R_1, I_1)$ to another $(R_2, I_2)$ is defined to be the orientation-preserving isomorphism from $R_1$ to $R_2$ and sending $I_1$ to $I_2$. 
\end{proposition}
\begin{proof}
From the construction above, the map is well defined. We first prove the surjectivity.  Given an oriented quadratic ring $R$ with $\mathrm{disc}(R) = D$ and an oriented ideal $I \subset R$ defined by:
$$R = \langle 1, \tau_D \rangle \ \mathrm{and} \ I = \langle \alpha, \beta \rangle,$$
where $\tau_D$ is defined in $(11)$ and the orientation of $I$ is determined by the ordered basis $[\alpha, \beta]$, we can always assume that the norm $\mathrm{N}(\alpha) = \alpha \cdot \alpha' \neq 0$. This is trivial in the number field case. To prove it in the general case, we define a binary quadratic form by 
\begin{equation*}\tag{12} (\alpha \cdot \alpha')  u^2 - (\alpha' \cdot \beta + \alpha \cdot \beta') uv +( \beta \cdot \beta') v^2 , \end{equation*}
then it is easy to see that the discriminant of $(12)$ is exactly the discriminant of $I$ given by
$$\mathrm{disc}(I) = \left( \det \left( \begin{matrix} \alpha & \beta \\
\alpha' & \beta' \end{matrix} \right) \right)^2.$$
As 
$$\mathrm{disc}(I) = (\mathrm{N}(I))^2 \cdot \mathrm{disc}(R),$$
and $R$ is non-degenerate, it implies that $\mathrm{disc}(I) \neq 0$. So at least one of the coefficients of $u^2$ and $v^2$ is non-zero. We can assume $\alpha \cdot \alpha' \neq 0$ by changing the order of $\alpha$ and $\beta$. As $\alpha, \beta \in I$, $\mathrm{N}(I) | \mathrm{N}(\alpha)$ and $\mathrm{N}(I) | \mathrm{N}(\beta)$, it follows that $\mathrm{N}(I)$ is the common factor of all coefficients of $(12)$. After canceling this common factor, we write it as
\begin{equation*} \tag{13} mu^2 + n uv + l v^2\end{equation*}
with $D = n^2- 4 m l $. So the quadratic ring $R(D)$ can be also written as 
$$R(D)  = [1, \tau_1],$$
where the choice of $\tau_1$ is specified and it satisfies
\begin{equation*}\tag{14} \tau_1^2 + n \tau_1 + m l =0. \end{equation*}
We write 
$$ ( \alpha, \beta ) =(1, \tau_1) \left( \begin{array}{cc} p &r \notag \\ q &s\end{array} \right).$$
Substituting $u= \beta $ and $v= \alpha$ into $(13)$, it becomes zero, by comparing it with the defining equation $(14)$ of $\tau_1$, then 
\begin{align*} p= ms + nq \ \mathrm{and}  \  r = - lq;  \quad \mathrm{or} \quad p= -ms \ \mathrm{and}  \  r = ns + lq.  \end{align*}
As $I$ is an ideal with cyclic quotient,  we must have $g.c.d.(q, s) =1$. Then by  the elementary divisor theorem, we can transform the matrix $\left( \begin{array}{cc} p &r \notag \\ q &s\end{array} \right)$ by a left multiplication in $B_2(\bZ)$ and a right multiplication in $\SL_2(\bZ)$ to the matrix has the form $\left( \begin{array}{cc} a & \ast \notag \\ 0 & 1 \notag \end{array} \right)$ with $a =\pm  \mathrm{N}(I)$. Therefore under a certain basis, $R$ and $I$ can be written as 
\begin{equation*}\tag{15}R(D) = \langle 1, \tau_2 \rangle  \ \mathrm{and} \ I =\langle a, \tau_2 \rangle,\end{equation*}
where the choice of $\tau_2$ is made such that the basis $[1, \tau_2]$ is positively oriented. With $\alpha$ replaced by $a$ and $\beta$ replaced by $\tau_2$, by $(12)$ there is a binary quadratic form. From the discussion above, all coefficients of it are divisible by $|a| = \mathrm{N}(I)$. After canceling this common factor, we have
\begin{equation*}  a u^2+ b u v  + c v^2,\end{equation*}
with $D = b^2 - 4 ac$. This is the required binary quadratic form. 

To prove the injectivity of the map, suppose that two binary quadratic forms 
$$Q_i(u,v) = a_i u^2 + b_i u v+ c_i v^2,$$ 
where $a_1=a_2 =a \neq 0$ and $D = b_i^2 - 4 a_i c_i \neq 0$ for $i =1, 2$, have the same image. We want to prove that $b_2 =b_1+ 2n a$ and $c_2= n^2a+ b_1n +c_1$ for some integer $n$. From the definition of the map, the oriented quadratic ring and the oriented ideal can be written as 
\begin{align}R_1 & = \langle 1, \tau_1\rangle , I_1= \langle a, \tau_1 \rangle  ,\notag \\
 R_2 &= \langle 1, \tau_2\rangle , I_2=\langle a, \tau_2 \rangle\notag \end{align} 
respectively, where the choice of $\tau_i's$ are made such that both bases $[1, \tau_1]$ and $[1, \tau_2]$ are positively oriented, and they satisfy
 $$\tau_i^2 + b_i \tau_i + a c_i =0.$$ 
As they have the same image, there exists an isomorphism $f$ from $R_1$ to $R_2$ preserving the orientation, so it has the form:
$$f(\tau_1)  =  \tau_2 +s ,$$
As it also satisfies $f(I_1 ) = I_2= \bZ a + \bZ \tau_2$, so we have
\begin{align}
s&= na, \notag \\
b_2 &= b_1+ 2 na, \notag \\
c_2 &= n^2a + b_1n + c_1, \notag
\end{align}
for some integer $n$. 
\end{proof}

Now we return to the case of $2 \times 2\times 2$ integer cubes. Giver a $2\times 2\times 2$ integer cube $A$, suppose that the $D = \mathrm{disc}(A) \neq 0$, we consider the two binary quadratic forms $Q^1_A(u,v)$ and $Q^2_A(u,v)$ associated to $A$. Suppose that the coefficients $a_i$ of $u^2$ are not zero, then applying the map in the last proposition to each $Q^i_A(u,v)$, we get the pairs $(R; I_1, I_2)$ where $I_i$ is the oriented ideal with $R / I_i \cong |a_i| \bZ$. Explicitly, the map is given by 
\begin{align*}\tag{16}
& \{A : Q^i_A(u,v) = a_i u ^2 + b_i uv + c_i v^2 , i=1,2 \} \notag \\& \rightarrow \{(R; I_1, I_2): I_i = \langle a_i, \tau_i \rangle, \tau_i ^2 + b_i \tau + a_i c_i =0  \}. \notag
\end{align*}

\begin{theorem}
With the notation above. Then the natural map of $(16)$ defines a surjective and finite morphism
$$\left( B_2(\bZ) \times B_2(\bZ) \times \SL_2(\bZ) \right) \backslash V^{ss}(\bZ) \to \mathrm{Iso} \backslash \{(R;I_1, I_2): R/I_1 \cong \mathrm{N}(I_1)\bZ, R/I_2 \cong \mathrm{N}(I_2)\bZ \}.$$
The cardinality $ n( R; I_1, I_2)$ of the fiber is equal to
$$\sigma_1( D_1, a_1, a_2),$$
where $D =D_0 D_1^2= \mathrm{disc}(R)$, and $D_0$ is square-free.
And it satisfies
$$\sum _{\substack{(R; I_1, I_2)/\sim \\ \mathrm{N}(I_i)=|a_i|}} n (R; I_1, I_2)  = B(D, |a_1|, |a_2|).$$
\end{theorem}
\begin{proof}
The map is described above and it is easy to see well defined. We first prove the surjectivity of the map. Given a pair $(R; I_1, I_2)$ with $I_i$ an oriented ideal of $R$ and $R/I_i  \cong \mathrm{N}(I_i) \bZ$, by Proposition 5.1, we know that there are two binary quadratic forms
$$Q_i(u,v) = a_i u^2+ b_i uv+ c_i v^2 $$
for $i =1, 2$ with $D= \mathrm{disc}(R) = b_i^2 - 4 a_i c_i$, such that $R$ and $I_i$ are determined by 
\begin{align}
R &= \langle 1, \tau_D \rangle  = \langle 1, \tau_1 \rangle = \langle 1, \tau_2 \rangle ,\notag \\
I_1 &= \langle a_1, \tau_1 \rangle \ \mathrm{and} \ I_2 = \langle a_2, \tau_2 \rangle ,\notag 
\end{align}
where $\tau_D$ is defined in $(11)$, and $\tau_i$ satisfies
$$ \tau_i^2 + b_i \tau_i  +a_i c_i =0$$
for $ i =1,2$.
By the Lemma 3.3, we know that there exists a $2 \times 2 \times 2$ integer cube $A$ such that
$$Q^i_A(u,v) = Q_i(u,v).$$
Under the correspondence of $(16)$, we conclude that the integer cube $A$ maps to the given pair $(R; I_1, I_2)$.  

To prove the second part of the theorem, let $A$ and $A'$ be the two $2\times 2 \times2$ integer cubes which are in the fiber of $(R; I_1, I_2)$ with $\mathrm{N}(I_i)=|a_i|$, then they have the following arithmetic property: $\mathrm{disc}(R) =\mathrm{disc}(A)  =\mathrm{disc}(A')=D $, and the two quadratic forms associated to them are the same $Q^i_A(u,v) = Q^i_{A'}(u,v) = a_i u^2 + b_i uv + c_i v^2$, where $0 \leq b_i \leq 2 |a_i|-1 $ is assumed. Write $D = D_0 D_1^2$. From our general formula of $B(D, m,n)$ in the Shintani zeta function $Z_{\rm{Shintani}}(s_1, s_2, w)$, we know that the fiber counting function is equal to
$$n(R; I_1, I_2)= \sum_{\substack{d|D_1\\ d|a_1, d|a_2}} f(d), \mathrm{where} \ f(d)  = \begin{cases} d  \ \mathrm{if} \ (\frac{b_i}{d})^2 \equiv \frac{D}{d^2} \ (\mathrm{mod}\   \frac{4a_i}{d}), \notag  \\ 0 \ \mathrm{otherwise} . \ \notag \end{cases}$$
Note that as $b_i^2 \equiv D \ ( \mathrm{mod} \ 4a_i  ) $ already holds, so it automatically implies  $(\frac{b_i}{d})^2 \equiv \frac{D}{d^2} \ (\mathrm{mod}\  \frac{4a_i}{d})$ if $d | g.c.d.(D_1, a_1, a_2)$. 
It follows that 
$$n(R; I_1, I_2) = \sigma_1\left( g.c.d.( D_1, |a_1|, |a_2|) \right).$$
Furthermore, the sum of cardinalities over the fixed norms of $\mathrm{N}(I_1) =| a_1|$ and $\mathrm{N}(I_2) =|a_2|$ is exactly
$$B(D, |a_1|, |a_2|) = \# \{ A \in V^{ss}_{\bZ}/\sim : \mathrm{disc}(A) =D,  |\det (M_A^1)| =|a_1|, |\det(M_A^2)| = |a_2|\}.$$
\end{proof}

\bibliography{mumford-tate}
\end{document}